\documentclass[journal,onecolumn,12pt]{article}

\usepackage{amssymb}
\usepackage{caption2}
\usepackage{amsmath}
\usepackage{multirow}
\usepackage{amsthm}
\usepackage{graphicx}
\usepackage{CJK}
\usepackage{enumerate}

\usepackage{tikz}
\usepackage{tikz}
\usetikzlibrary{positioning,chains,fit,shapes,calc}

\newtheorem{theorem}{Theorem}[section]
\newtheorem{lemma}[theorem]{Lemma}
\newtheorem{proposition}[theorem]{Proposition}

\newtheorem{conjecture}[theorem]{Conjecture}

\newtheorem{remark}[theorem]{Remark}

\begin{document}
\title{On the minimal degree condition of graphs implying some properties of subgraphs}

\author{Bingchen Qian$^{\text{a}},$ Chengfei Xie$^{\text{b}}$ and Gennian Ge$^{\text{b,}}$\thanks{
  The research of G. Ge was supported by the National Natural Science Foundation of China under Grant No. 11971325, National Key Research and Development Program of China under Grant  Nos. 2020YFA0712100  and  2018YFA0704703,  and Beijing Scholars Program.}\\
\footnotesize $^{\text{a}}$ School of Mathematical Sciences, Zhejiang University, Hangzhou 310027, Zhejiang, China\\
\footnotesize $^{\text{b}}$ School of Mathematical Sciences, Capital Normal University, Beijing, 100048, China}

\maketitle

\begin{abstract}
Erd\H{o}s posed the problem of finding conditions on a graph $G$ that imply the largest number of edges in a triangle-free subgraph is equal to the largest number of edges in a bipartite subgraph. We generalize this problem to general cases. Let $\delta_r$ be the least number so that any graph $G$ on $n$ vertices with minimum degree $\delta_rn$ has the property $P_{r-1}(G)=K_rf(G),$ where $P_{r-1}(G)$ is the largest number of edges in an $(r-1)$-partite subgraph and $K_rf(G)$ is the largest number of edges in a $K_r$-free subgraph. We show that $\frac{3r-4}{3r-1}<\delta_r\le\frac{4(3r-7)(r-1)+1}{4(r-2)(3r-4)}$ when $r\ge4.$ In particular, $\delta_4\le 0.9415.$


  \smallskip

\smallskip
\end{abstract}

\section{Introduction}
In $1983,$ Erd\H{o}s \cite{MR777160} raised the question that for which graphs do the largest bipartite subgraph and the largest triangle-free subgraph have the same number of edges. He noted that the equality holds for the complete graph $K_n$ by Tur\'{a}n's theorem. Later, Babai, Simonovits and Spencer \cite{MR1073101} proved that the equality holds almost surely for the random graph $G(n,1/2),$ whose edges are chosen with probability $1/2.$ A general condition was given by Bondy, Shen, Thomass\'{e} and Thomassen \cite{MR2223630} showing that a minimum degree condition is sufficient. In $2006,$ Balogh, Keevash and Sudakov \cite{MR2274084} improved both the upper and the lower bounds for the minimum degree condition implying equality.

For a graph $G,$ we write $P_r(G)$ for the largest number of edges in an $r$-partite subgraph, and $K_rf(G)$ for the largest number of edges in a $K_r$-free subgraph for each $r\ge3$. Clearly, $P_{r-1}(G)\le K_{r}f(G).$ Let $\delta_r$ denote the least number such that, for $n$ sufficiently large, any graph $G$ on $n$ vertices with minimum degree $\delta(G)\ge\delta_r n$ has the property $P_{r-1}(G)=K_{r}f(G).$ Bondy et al. \cite{MR2223630} showed that $0.675\le\delta_3\le 0.85$ and later Balogh et al. \cite{MR2274084} improved the result to $0.75\le\delta_3<0.791.$

In this paper, we give both an upper bound and a lower bound for general $\delta_r$ with $r\ge4$ as follows.
\begin{theorem}\label{general upper bound}
  When $r\ge4,$ we have $\frac{3r-4}{3r-1}<\delta_r\le\frac{4(3r-7)(r-1)+1}{4(r-2)(3r-4)}.$
\end{theorem}

For the case $r=4,$ we also derive a better upper bound.
\begin{theorem}\label{upper bound for r=4}
  $\delta_4\le0.9415.$
\end{theorem}

This paper is organised as follows. In the next section, we will give a proof of the general upper bound for $\delta_r$ in Theorem \ref{general upper bound}. In Section $3$, we will describe several inductive constructions to prove the lower bound. In Section $4,$ we show some important properties of $K_4$-free graphs under some minimum degree conditions and prove a better upper bound for $\delta_4.$ Section $5$ contains a little stronger upper bound for $\delta_4.$ The last section contains some concluding remarks.

\section{A general upper bound for $\delta_r$}

\textbf{Notations.} Let $G=(V,E)$ be a graph with vertex set $V(G)$ and edge set $E(G).$ Write $v(G)=|V(G)|$ and $e(G)=|E(G)|.$ If $X\subset V(G),$ then $G[X]$ denotes the restriction of $G$ to $X,$ i.e. the graph with vertex set $X$ whose edges are those edges of $G$ with both endpoints in $X.$
Given a graph $G$ and a vertex $u$ in $V(G),$ $d_A(u)$ denotes the neighbors of $u$ in $A,$ where $A\subseteq V(G).$
 For $k\ge2,$ the $k$-th power of $G,$ $G^k,$ is a graph such that $V(G^k)=V(G)$ and $E(G^k)=\{uv: \text{distance between } u\text{ and }v\text{ in } G\text{ is at most }k\}.$ For two graphs $G$ and $H,$ $G+H$ is a graph such that $V(G+H)=V(G)\cup V(H)$ and $E(G+H)=E(G)\cup E(H)\cup\{uv: u\in V(G), v\in V(H)\}.$

Before we prove the upper bound, we need some preparations.

The following lemma is an easy exercise for the probabilistic method.
\begin{lemma}
  Suppose $G$ is a graph on $n$ vertices, then $G$ contains an $(r-1)$-partite subgraph with at least $\frac{r-2}{r-1}e(G)$ edges.
\end{lemma}
\begin{proof}
  Consider a random partition of vertices of $G$ into $r-1$ parts and the probability of each edge that still survives is $1-\frac{r-1}{(r-1)^2}=\frac{r-2}{r-1}.$ By the linearity of expectation, there exists an $(r-1)$-partite subgraph with at least $\frac{r-2}{r-1}e(G)$ edges.
\end{proof}

\begin{lemma}\label{r-1 partite wiht more edges}
  Suppose $\Gamma$ is an $(r-1)$-partite subgraph of a graph $G$ and there are $m$ edges incident to the vertices in $V(G)\setminus V(\Gamma),$ then $G$ has an $(r-1)$-partite subgraph of size at least $e(\Gamma)+\frac{r-2}{r-1}m.$
\end{lemma}

\begin{proof}
  Let $(A_1,A_2,\dots,A_{r-1})$ be the partition of $\Gamma.$ Consider an $(r-1)$-partite subgraph $G'$ with parts $(B_1,B_2,\dots,B_{r-1})$ where $A_i\subseteq B_i$ for $i\in\{1,2,\dots,r-1\}.$ Then we place each vertex $v\in V(G)\setminus V(\Gamma)$ in $B_i$ independently with probability $\frac{1}{r-1}$ for each $1\le i\le r-1$, so each edge incident to a vertex in $V(G)\setminus V(\Gamma)$ appears in $G'$ with probability $\frac{r-2}{r-1}.$ By the linearity of expectation $E[e(G')]=e(\Gamma)+\frac{r-2}{r-1}m,$ there exists an $(r-1)$-partite subgraph of $G$ with at least the desired number of edges, completing the proof.
\end{proof}

\begin{remark}
  In \cite{MR2274084}, Balogh et al. proved the bipartite case. The proof for general cases is similar. For the sake of completeness, we give a proof here.
\end{remark}

The following lemma describes the sufficiency when a $K_r$-free graph must be $r-1$ partite.
\begin{lemma}[\cite{MR340075}]\label{Kr-free r-1 partite}
  Let $G$ be a $K_r$-free graph on $n$ vertices such that $\delta(G)>\frac{3r-7}{3r-4}n.$ Then $\chi(G)\le r-1.$
\end{lemma}

Now we restate Theorem \ref{general upper bound} and prove it. The proof is similar to that for the case of $r=3$ in \cite{MR2223630}.
\begin{theorem}
  Let $G$ be a graph on $n$ vertices with minimum degree $\delta(G)\ge\frac{4(3r-7)(r-1)+1}{4(r-2)(3r-4)}n+1,$ where $n$ is sufficiently large. Then the largest $K_r$-free and the largest $(r-1)$-partite subgraphs of $G$ have equal size. Therefore, $\delta_r\le\frac{4(3r-7)(r-1)+1}{4(r-2)(3r-4)}.$
\end{theorem}
\begin{proof}
  Set $\delta:=\frac{4(3r-7)(r-1)+1}{4(r-2)(3r-4)}$ and suppose $G$ is an $n$-vertex graph with minimum degree $\delta(G)\ge\delta n+1.$ Let $H$ be the largest $K_r$-free subgraph in $G$ and $A$ the largest $(r-1)$-partite subgraph in $G$. By an easy application of the probabilistic method, we have that $e(H)\ge e(A)\ge\frac{r-2}{r-1}e(G)\ge\frac{r-2}{r-1}\frac{\delta(G)}{2}n\ge\frac{r-2}{r-1}\frac{\delta n+1}{2}n.$ Now, if $H$ has a vertex $x_1$ of degree at most $\frac{3r-7}{3r-4}n$, we delete it. Similarly, if $H\setminus x_1$ has a vertex $x_2$ with degree at most $\frac{3r-7}{3r-4}(n-1),$ we delete it again. We continue in this way until we obtain a graph $M$ with $m$ vertices and minimum degree greater than $\frac{3r-7}{3r-4}m.$ By Lemma \ref{Kr-free r-1 partite} , $M$ is $(r-1)$-partite. The claim $e(H)=e(A)$ follows trivially if $m=n.$ Now suppose $m<n.$ As
  \begin{align*}
    {{r-1}\choose{2}}(\frac{m}{r-1})^2=\frac{m^2(r-2)}{2(r-1)}\ge& e(M)\ge e(H)-\frac{3r-7}{3r-4}({n+1\choose 2}-{m+1\choose 2})\\
    \ge&\frac{r-2}{r-1}\frac{\delta}{2}n^2-\frac{3r-7}{2(3r-4)}(n^2-m^2),
  \end{align*}
it follows that $m\ge(\delta(r-2)(3r-4)-(3r-7)(r-1))^{1/2}n.$\\
Let $V(G)\setminus V(M)=\{x_1, x_2,\dots,x_{n-m}\}.$ Define $F=\{e\in E(G): e \text{ is incident to at least one of } x_i,~ 1\le i\le n-m\}.$ Then
\[
    |F|\ge\sum_{i=1}^{n-m}d_{G}(x_i)-{n-m\choose 2}\ge(\delta n+1)(n-m)-{n-m\choose 2}.
\]
By Lemma \ref{r-1 partite wiht more edges}, there exists an $(r-1)$-partite graph such that the number of its edges is at least
\begin{align*}
    e(H)-\frac{3r-7}{3r-4}({n+1\choose 2}-{m+1\choose 2})+\frac{r-2}{r-1}|F|\ge&e(H)-\frac{3r-7}{3r-4}({n+1\choose 2}-{m+1\choose 2})\\
     &+\frac{r-2}{r-1}((\delta n+1)(n-m)-{n-m\choose 2}).
\end{align*}
If $e(H)-\frac{3r-7}{3r-4}({n+1\choose 2}-{m+1\choose 2})+\frac{r-2}{r-1}((\delta n+1)(n-m)-{n-m\choose 2})\ge e(H),$ which implies $e(A)\ge e(H)$ and thus $e(A)=e(H),$ we are done.\\
That is
\[
    \frac{r-2}{r-1}((\delta n+1)(n-m)-{n-m\choose2})-\frac{3r-7}{3r-4}({n+1\choose2}-{m+1\choose2})\ge0.
\]
Since $m<n,$ we divide both sides by $n-m$ and we have that
\[
    m\ge 2(r-1)(3r-4)((\frac{1}{2}-\delta)\frac{r-2}{r-1}+\frac{3r-7}{2(3r-4)})n-\frac{r-2}{2(r-1)}.
\]
The only condition on $m$ is that $n>m\ge(\delta(r-2)(3r-4)-(3r-7)(r-1))^{1/2}n,$ hence it suffices to show
\[
  (\delta(r-2)(3r-4)-(3r-7)(r-1))^{1/2}n\ge 2(r-1)(3r-4)((\frac{1}{2}-\delta)\frac{r-2}{r-1}+\frac{3r-7}{2(3r-4)})n.
\]
Now we substitute $\delta=\frac{4(3r-7)(r-1)+1}{4(r-2)(3r-4)}$ in the above inequality, then the conclusion follows.
\end{proof}

\section{A general lower bound for $\delta_r$}

In this section, we give a lower bound for $\delta_r$ when $r\ge4.$ We first prove the case when $r=4$ to illustrate our main idea, which is also our base case.

In \cite{MR2274084}, Balogh et. al. gave the following lower bound for $\delta_3.$
\begin{theorem}[\cite{MR2274084}]
For any $\delta<\frac{3}{4},$ there is $n$ and a graph $G$ on $n$ vertices with minimum degree
at least $\delta n,$ in which the largest triangle-free subgraph has more edges than the largest bipartite subgraph. Therefore $\delta_3\ge\frac{3}{4}.$
\end{theorem}
Based on their construction, we give a lower bound for $\delta_4.$
\begin{theorem}\label{lower bound for delta 4}
$\delta_4>\frac{8}{11}.$
\end{theorem}
\begin{proof}
   Let $G$ be a graph whose $n$ vertices are partitioned into $6$ parts $V_0, V_1, \dots, V_5$ such that $v(V_i)=\frac{8}{55}n$ for $i\in\mathbb{Z}_5=\{0,1,2,3,4\}$ and $v(V_5)=\frac{3}{11}n.$ Let $G':=G[V(G)\backslash V_5]$ whose structure looks exactly same as the construction given by \cite{MR2274084}. That is, $G'[V_i\cup V_{i+1}]$ forms a clique for $i\in\mathbb{Z}_5$ and each pair $uv$ with $u\in V_i$ in $G',$ $v\in V_{i+2}$ in $G'$ is chosen to be an edge randomly and independently with probability $\theta,$ for some $\theta < \frac{3}{8}.$ By the probability method, the authors in \cite{MR2274084} showed the existence of a graph with the probability that $e(A_i,A_{i+2})=\theta|A_i||A_{i+2}|+o(n^2)$ for any two subsets $A_i\subset V_i,$ $A_{i+2}\subset V_{i+2},$ where $i, i+2\in\mathbb{Z}_5.$ We denote it by $G'$ for simplicity. So $d_{G'}(v)=\frac{8}{55}(3+2\theta)n+o(n)$ for any $v\in G'.$ The vertices of $V_5$ form an independent set which is connected to all vertices outside of $V_5.$ So  $d_G(v)=(\frac{8}{55}(3+2\theta)+\frac{3}{11})n+o(n)>\frac{8}{11}n$ for $v\in V_i$ when $\theta$ is close to $\frac{1}{8}$, $i\in\mathbb{Z}_5$ and $d_G(u)=\frac{8}{11}n$ for $u\in V_5.$ That is, $\delta(G)=\frac{8}{11}n.$

  What we need to show now is that the largest $K_4$-free subgraph of $G$ has more edges than the largest $3$-partite subgraph.

  Assume the largest $3$-partite subgraph of $G$ partitions $V(G)$ into three parts $C_1,$ $C_2$ and $C_3.$ Let $B_i=V(G')\cap C_i$ for $1\le i\le3,$ and without loss of generality we assume that $|B_1|\le|B_2|\le|B_3|.$ Since the vertices in $V_5$ form an independent set and connect to every vertex outside of $V_5,$ we put $V_5$ in part $C_1$ which will not decrease the number of edges. Now consider the vertices in $B_1.$ Let $v\in B_1,$ by the structure of $G',$ we have $d_{G'}(v)<\frac{6}{11}n.$  There is an $i\in\{2, 3\}$ such that if we remove $v$ to $C_i,$ then it will decrease less than $\frac{3}{11}n$ edges, but at the same time it will get $\frac{3}{11}n$ new edges since $v$ is connected to every vertex in $V_5.$ Therefore, if $B_1$ is not empty, then we can always move a vertex in $B_1$ to $C_2$ or $C_3$ and get a strictly larger $3$-partite subgraph. So $B_1$ is empty. Based on the above argument, the structure of the largest $3$-partite subgraph must be that all vertices of $V_5$ form $C_1$ and the vertices of $G'$ form $C_2$ and $C_3$. By \cite{MR2274084}, the largest bipartite subgraph of $G'$ has less than $\frac{64}{605}n^2$ edges. So $P_3(G)<\frac{184}{605}n^2.$ However, the subgraph of $G$ obtained by removing the edges inside each $V_i$ for $i\in\mathbb{Z}_5$ is $K_4$-free and has $\frac{184}{605}n^2$ edges. We are done.
\end{proof}

Based on the above theorem, now we can derive the lower bounds for general cases.

\begin{theorem}
$\delta_r>\frac{3r-4}{3r-1}$ for $r\geq4$.
\end{theorem}
\begin{proof}
Let $G_4$ be the graph $G$ in the proof of Theorem \ref{lower bound for delta 4}. We define a sequence of graphs $\{G_r\}_{r=4}^{\infty}$ recursively.

Suppose we have defined $G_{r-1}$, such that $V(G_{r-1})=V_0\cup V_1\cup\cdots\cup V_r$, where $|V_0|=|V_1|=\cdots=|V_4|=\frac{8}{55}n, |V_5|=|V_6|\cdots=|V_r|=\frac{3}{11}n$. Moreover, suppose $P_{r-2}(G_{r-1})<K_{r-1}f(G_{r-1})$.

Then we define $G_r$ as follows: $V(G_{r})=V_0\cup V_1\cup\cdots\cup V_{r+1}$, where $|V_{r+1}|=\frac{3}{11}n$, and $E(G_r)=E(G_{r-1})\cup\{uv:u\in V_0\cup V_1\cup\cdots\cup V_r, v\in V_{r+1}\}$. We view $G_{r-1}$ as an induced subgraph of $G_r$.

Then
$$
d_{G_r}(v)=\frac{8}{55}(3+2\theta)n+\frac{3}{11}(r-3)n+o(n)=(\frac{3}{11}r-\frac{1}{11})n+\frac{8}{55}(2\theta-2)n+o(n), \forall v\in V_0\cup V_1\cup\cdots\cup V_4;
$$
$$
d_{G_r}(u)=\frac{8}{11}n+\frac{3}{11}(r-4)n=(\frac{3}{11}r-\frac{4}{11})n, \forall u\in V_5\cup\cdots\cup V_{r+1}.
$$

In particular, we can take $\theta=\frac{1}{16}$, then
$$
d_{G_r}(v)=(\frac{3}{11}r-\frac{4}{11})n+o(n), \forall v\in V_0\cup V_1\cup\cdots\cup V_4.
$$

Assume the largest $(r-1)$-partite subgraph of $G_r$ partitions $V(G_r)$ into $(r-1)$ parts, say $C_1, C_2,\cdots, C_{r-1}$. Since the vertices in $V_{r+1}$ form an independent set and connect to every vertex outside of $V_{r+1}$, putting $V_{r+1}$ in part $C_{r-1}$ will not decrease the number of edges. Now consider the rest of vertices.

$\forall u\in V_5\cup\cdots\cup V_{r},$ $ d_{G_{r-1}}(u)=(\frac{3}{11}r-\frac{7}{11})n$ by induction. If there exists such a vertex $u\in C_{r-1},$ then there is $i\in\{1, 2, \cdots, r-2\}$, such that $d_{C_i}(u)\le\frac{1}{r-2}(\frac{3}{11}r-\frac{7}{11})n$ and if we move $u$ to $C_i,$ it will decrease at most $\frac{1}{r-2}(\frac{3}{11}r-\frac{7}{11})n$ edges, but at the same time, it will get $\frac{3}{11}n$($>\frac{1}{r-2}(\frac{3}{11}r-\frac{7}{11})n$) new edges since $u$ is connected to every vertex in $V_{r+1}\subseteq C_{r-1}$. Thus there would be no vertex of $V_5\cup\cdots\cup V_{r}$ in $C_{r-1},$ otherwise we will get a strictly larger $(r-1)$-partite subgraph.

$\forall v\in V_0\cup V_1\cup\cdots\cup V_4,$ $ d_{G_{r-1}}(v)=(\frac{3}{11}r-\frac{7}{11})n+o(n)$ by induction. Similarly, if there exists such a $v\in C_{r-1},$ then there is $i\in\{1, 2, \cdots, r-2\}$, such that $d_{C_i}(v)=\frac{1}{r-2}((\frac{3}{11}r-\frac{7}{11})n+o(n))$ and if we move $v$ to $C_i,$ it will decrease at most $\frac{1}{r-2}((\frac{3}{11}r-\frac{7}{11})n+o(n))$ edges, but at the same time, it will get $\frac{3}{11}n$($>\frac{1}{r-2}[(\frac{3}{11}r-\frac{7}{11})n+o(n)]$) new edges since $v$ is connected to every vertex in $V_{r+1}\subseteq C_{r-1}$. Thus there would be no vertex of $V_0\cup V_1\cup\cdots\cup V_4$ in $C_{r-1},$ otherwise we will get a strictly larger $(r-1)$-partite subgraph.

By the above argument, we conclude that the structure of the largest $(r-1)$-partite subgraph of $G_r$ must be that all vertices of $V_{r+1}$ form $C_{r-1}$ and the vertices of $V(G_{r-1})$ form the other $(r-2)$ parts. So $P_{r-1}(G_{r})=P_{r-2}(G_{r-1})+|V_{r+1}||V(G_{r-1})|<K_{r-1}f(G_{r-1})+|V_{r+1}||V(G_{r-1})|\leq K_{r}f(G_{r})$.

We have shown that the minimal degree of $G_r$ is $(\frac{3}{11}r-\frac{4}{11})n$, and the number of vertices in $G_r$ is $(\frac{3}{11}r-\frac{1}{11})n$. Therefore, for every $r\geq4$, we have constructed a graph $G_r$, in which the ratio of the minimal degree and the number of vertices is $\frac{3r-4}{3r-1}$, with $P_{r-1}(G_{r})<K_{r}f(G_{r})$. Hence, $\delta_r>\frac{3r-4}{3r-1}$.
\end{proof}

\section{A weak upper bound for $\delta_4$}

The techniques employed in this section and the next one   are adapted from \cite{MR2274084}, but  some more complicated analysis should be involved here. Before we prove Theorem \ref{upper bound for r=4}, it will be helpful to first give a slightly weak upper bound.

We first describe the structure of $K_r$-free graphs with large minimal degrees. For $d\ge 1$ we define a graph $F_d$ as follows: set $F_1=K_2$ and for every $d\ge2$ let $F_d$ be the complement of the $(d-1)$-th power of the cycle $C_{3d-1}.$ To be more precise, the vertex set $V(F_d)$ consists of the integers modulo $3d-1,$ which is denote by $\mathbb{Z}_{3d-1}.$ The vertex $v\in\mathbb{Z}_{3d-1}$ is adjacent to the vertices $v+1, v+4,v+7,\dots,v-1.$ Thus $F_d$ is a $d$-regular graph on $3d-1$ vertices.

A graph $G$ is said to be homomorphic to a graph $H,$ if there is a map $f:V(G)\rightarrow V(H)$ such that $uv\in E(G)$ implies that $f(u)f(v)\in E(H).$

In \cite{MR1264720}, Jin generalized the case $r=2$ of the theorem of Andr\'{a}sfai, Erd\H{o}s and S\'{o}s and a result of H\"{a}ggkvist from \cite{MR671908}.
\begin{theorem}[\cite{MR1264720}]
  Let $1\le d\le 9,$ and let $G$ be a triangle-free graph of order $n$ with minimum degree $\delta>\frac{d+1}{3d+2}n,$ then $G$ is homomorphic to $F_d.$
\end{theorem}

In Chapter $4$ of \cite{MR2866729}, Nikiforov proved a general result for $K_r$-free graphs which will be used later for the case $r=3$.
\begin{theorem}[\cite{MR2866729} Chapter $4,$ Theorem 2.28]\label{homomorphic}
  Let $r\ge2,$ $1 \le d \le 9,$ and let $G$ be a $K_{r+1}$-free graph of order $n$. If $\delta(G) > (1-\frac{2d-1}{(2d-1)r-d+1})n,$ then $G$ is homomorphic to $F_{d} + K_{r-2}.$
\end{theorem}

Next we will need a lemma which describes the properties of these graphs under certain minimum degree conditions.
\begin{lemma}\label{min lemma}
Suppose $d\ge 2$ and the vertices of $F_d+K_1$ are weighted by reals, we label the vertices in $F_d$ by $\{0,1,\dots, 3d-2\}$ and the vertex in $K_1$ by $3d-1,$ such that vertex i has weight $x_i$, where $0\le x_i\le 1$ and $\sum_{i=0}^{3d-1}x_i=1$. Write 
  $g_i=\sum_{j:j\sim i}x_j$ and $e=\frac{1}{2}\sum_{i}x_{i}g_{i}.$ Suppose that $g_i\ge \gamma$ for each $i\in\mathbb{Z}_{3d}.$ Then
\begin{align*}
    \gamma &\le\frac{3d-1}{5d-2},\\
    e&\le\frac{1}{6}(125d^2\gamma^2-150d^2\gamma+45d^2-175d\gamma^2+200d\gamma-57d+50\gamma^2-50\gamma+14).
\end{align*}
\end{lemma}
\begin{proof}
   Note that the vertex $3d-1$ is adjacent to all vertices in $F_d$ and every $i\in\mathbb{Z}_{3d-1}$ is adjacent to exactly one element of $\{0, 1, 2\},$ apart from $1$, which is adjacent to both $0$ and $2.$ Therefore,
  \begin{align*}
   5\gamma &\le g_0+g_1+g_2+2g_{3d-1}\\
    &=x_1+\sum_{i=0}^{3d-2}x_i+3x_{3d-1}+2\sum_{i=0}^{3d-2}x_i\\
    &=x_1+3,
  \end{align*}

  so $x_1\ge5\gamma -3.$ Similarly, $x_i\ge5\gamma-3$ for $i\in\{0,1,\ldots,3d-2\}.$ Also we have
  \begin{align*}
    (5d-7)\gamma&\le g_3+g_4+\cdots+g_{3d-2}+(2d-3)g_{3d-1}\nonumber\\
        &=\sum_{i=0}^{3d-2}g_i-(g_0+g_1+g_2)+(2d-3)\sum_{i=0}^{3d-2}x_i \\
        &=d\sum_{i=0}^{3d-2}x_i+(3d-1)x_{3d-1}-x_1-3+2\sum_{i=0}^{3d-2}x_i+(2d-3)\sum_{i=0}^{3d-2}x_i\\
        &=3d-4-x_1,
  \end{align*}
  so $x_1\le 3d-4-(5d-7)\gamma.$ Combining these two inequalities, we have $3d-1\ge(5d-2)\gamma.$ Set
  \[
    y_i=\frac{x_i-(5\gamma -3)}{3d-1-\gamma(5d-2)}, ~~~~  for ~ i=0,1,\dots,3d-2
  \]
  and
  \[
    y_{3d-1}=\frac{x_{3d-1}-(1-\gamma)}{3d-1-\gamma(5d-2)}.
  \]
  In case of $3d-1-\gamma(5d-2)=0$, we will have that $x_i-(5\gamma -3)=0$ for $i=0,1,\dots,3d-2$ and $x_{3d-1}-(1-\gamma)=0$, so we set $y_i=0$ for all $i$ and it will cause no confusion. Then $0\le y_i\le 1$ for $i\in\{0,1,\dots,3d-2\}$ and $\frac{-(1-\gamma)}{3d-1-\gamma(5d-2)}\le y_{3d-1}\le 0,$ and
  \[
    \sum_{i=0}^{3d-1}y_i = \frac{\sum_{i=0}^{3d-1}x_i-(3d-1)(5\gamma-3)-1+\gamma}{3d-1-\gamma(5d-2)}=3.
  \]
  For convenience, we set $t=3d-1-\gamma(5d-2)\ge0,$ and write
  \begin{align*}
    e =& \sum_{i\sim j}x_i x_j = \sum_{\substack{i\sim j\\i,j=0}}^{3d-2}x_ix_j +x_{3d-1}\sum_{i=0}^{3d-2}x_i \\
    =&\sum_{\substack{i\sim j\\i,j=0}}^{3d-2}(5\gamma-3+ty_i)(5\gamma-3+ty_j)+(1-\gamma+ty_{3d-1})\sum_{i=0}^{3d-2}(5\gamma-3+ty_i)\\
    =&\sum_{\substack{i\sim j\\i,j=0}}^{3d-2}(5\gamma-3)^2+\sum_{i=0}^{3d-2}y_i\sum_{\substack{j:j\sim i\\j=0}}^{3d-2}(5\gamma-3)t \\
    &+t^2\sum_{\substack{i\sim j\\i,j=0}}^{3d-2}y_iy_j+(1-\gamma+ty_{3d-1})[(3d-1)(5\gamma-3)+t(3-y_{3d-1})]  \\
    =&\frac{1}{2}d(3d-1)(5\gamma-3)^2+(5\gamma-3)td(3-y_{3d-1})\\
    &+t^2\sum_{\substack{i\sim j\\i,j=0}}^{3d-2}y_iy_j+(1-\gamma)[(3d-1)(5\gamma-3)+3t]+t[(3d-1)(5\gamma-3)+3t]y_{3d-1}\\
    &-(1-\gamma)ty_{3d-1}-t^2y_{3d-1}^{2}\\
    =&\frac{1}{2}d(3d-1)(5\gamma-3)^2+3dt(5\gamma-3)+(1-\gamma)[(3d-1)(5\gamma-3)+3t]\\
    &-t^2y_{3d-1}^{2}+(t((3d-1)(5\gamma-3)+3t)-dt(5\gamma-3)-(1-\gamma)t)y_{3d-1}+t^2\sum_{\substack{i\sim j\\i,j=0}}^{3d-2}y_iy_j.
  \end{align*}
  Since $\sum_{i=0}^{3d-2}y_i=3-y_{3d-1}, y_i\ge 0$ for $i\in \{0,1,\dots,3d-2\}$, it is known in \cite{MR2274084} that the maximum of $\sum_{\substack{i\sim j\\i,j=0}}^{3d-2}y_iy_j$ subject only to the previous condition is achieved when the vertices with $y_i>0$ form a clique in the graph. $F_d$ is triangle free, so this clique is just an edge.
  Then $\sum_{\substack{i\sim j\\i,j=0}}^{3d-2}y_iy_j\le(\frac{3-y_{3d-1}}{2})^2.$ Therefore
  \begin{align*}
    e \le& \frac{1}{2}d(3d-1)(5\gamma-3)^2+3dt(5\gamma-3)+(1-\gamma)[(3d-1)(5\gamma-3)+3t]+\frac{9}{4}t^2\\
     &- \frac{3}{4}t^2y_{3d-1}^2  \\
     &+ (t((3d-1)(5\gamma-3)+3t)-dt(5\gamma-3)-(1-\gamma)t-\frac{3}{2} t^2)y_{3d-1}.
  \end{align*}
  Substituting $t=3d-1-\gamma(5d-2)$ into the above inequality and a careful calculating shows that the coefficient of $y_{3d-1}$ is $-\frac{1}{2}t^2$. So we have
  \begin{align*}
    e\le&\frac{1}{2}d(3d-1)(5\gamma-3)^2+3dt(5\gamma-3)+(1-\gamma)[(3d-1)(5\gamma-3)+3t]+\frac{9}{4}t^2\\
      &-\frac{1}{4}t^2(3y_{3d-1}^2+2y_{3d-1})\\
      =&\frac{1}{4}(75d^2\gamma^2-90d^2\gamma+27d^2-110d\gamma^2+126d\gamma-36d+32\gamma^2-32\gamma+9)\\
      &-\frac{1}{4}t^2(3y_{3d-1}^2+2y_{3d-1}).
  \end{align*}
  Since $\frac{-(1-\gamma)}{3d-1-\gamma(5d-2)}\le y_{3d-1}\le 0,$ we get that $-\frac{1}{4}t^2(3y_{3d-1}^2+2y_{3d-1})\le\frac{1}{12}t^2.$ All in all, we have that $e\le\frac{1}{6}(125d^2\gamma^2-150d^2\gamma+45d^2-175d\gamma^2+200d\gamma-57d+50\gamma^2-50\gamma+14).$\\

  For the case $d=2$, we can improve the upper bound slightly. By Theorem \ref{homomorphic}, $G$ is homomorphic to $F_2+K_1,$ which is a $5$-wheel. Set $z_i=1-y_i$ for $i\in\{0,1,\dots,4\},$ so that $0\le z_i\le 1$ and $\sum_{i=0}^{4}z_i=5-\sum_{i=0}^{4}y_i=5-(3-y_5)=2+y_5.$ Then
  \[
    \sum_{\substack{i\sim j\\i,j=0}}^{4}y_iy_j=\sum_{\substack{i\sim j\\i,j=0}}^{4}(1-z_i)(1-z_j)=5-2\sum_{i=0}^{4}z_i+\sum_{\substack{i\sim j\\i,j=0}}^{4}z_iz_j=1-2y_5+\sum_{\substack{i\sim j\\i,j=0}}^{4}z_iz_j.
  \]
  By the above argument $\sum_{\substack{i\sim j\\i,j=0}}^{4}z_iz_j\le(\frac{2+y_5}{2})^2=1+y_5+\frac{1}{4}y_5^2,$ we have$\sum_{\substack{i\sim j\\i,j=0}}^{4}y_iy_j\le 2-y_5+\frac{1}{4}y_5^2.$ Therefore,
  \begin{align*}
    e \le& \frac{1}{2}d(3d-1)(5\gamma-3)^2+3dt(5\gamma-3)+(1-\gamma)[(3d-1)(5\gamma-3)+3t]+2t^2\\
      &-\frac{3}{4}t^2y_5^2\\
      \le&\frac{1}{2}(25d^2\gamma^2-30d^2\gamma+9d^2-45d\gamma^2+52d\gamma-15d+14\gamma^2-14\gamma+4).
  \end{align*}
  Substituting $d=2$ into the above inequality gives $e\le12\gamma^2-15\gamma+5.$ This completes the proof.

\end{proof}

In order to derive Lemma \ref{condition of d} from Lemma \ref{min lemma}, we need the following proposition.
\begin{proposition}
  Let $d\ge 2$ and suppose that a graph $G$ is homomorphic to $F_d+K_1,$ with parts $\{V_0,V_1.\dots,V_{3d-2}\}+V_{3d-1},$ but not homomorphic to $F_i+K_1$ for any $i<d,$  then the homomorphic is surjective.
\end{proposition}
\begin{proof}
  Let $f$ be a homomorphism from $G$ to $F_d+K_1.$ Assume that $f$ is not surjective, then there exists some $i\in\{0,1,\dots,3d-1\}$ such that $i$ has no preimage.
  \begin{itemize}
    \item If $i\in\{0,1,\dots,3d-2\},$ we set $i=3d-2$ by symmetry. Now define
    \begin{equation*}
      g:(F_d+K_1)\backslash \{3d-2\}\rightarrow F_{d-1}+K_1
    \end{equation*}
    such that $g(j)=j$ for $0\le j\le 3d-5,$ $g(3d-4)=0,$ $g(3d-3)=1,$ and $g({3d-1})=3(d-1)-1.$ It is easy to verify that $g$ is a homomorphism. So $g\circ f$ is a homomorphism from $G$ to $F_{d-1}+K_1,$ which is a contradiction.
    \item If $i=3d-1,$ we define
    \begin{equation*}
      g:(F_d+K_1)\backslash \{3d-1\}\rightarrow F_{d-1}+K_1
    \end{equation*}
    such that $g(j)=j$ for $0\le j\le 3d-5,$ $g(3d-4)=0,$ $g(3d-3)=1$ and $g({3d-1})={3(d-1)-1}.$ It is also easy to check that $g$ is a homomorphism. Similarly, $g\circ f$ is a homomorphism from $G$ to $F_{d-1}+K_1,$ which is a contradiction.
  \end{itemize}
\end{proof}

The following lemma follows directly.
\begin{lemma}\label{condition of d}
  Let $G$ be a graph on $n$ vertices with minimum degree $\delta(G)\ge\gamma n$. Suppose that $G$ is homomorphic to $F_d+K_1$, with parts $\{V_0,V_1.\dots,V_{3d-2}\}+V_{3d-1},$ but not homomorphic to $F_i+K_1$ for any $i<d,$  then
  \begin{enumerate}[(1)]

    \item If $d=2$ then $\gamma\le\frac{5}{8}$ and $n^{-2}e\le12\gamma^2-15\gamma+5.$ 
    \item If $d=3$ then $\gamma\le\frac{8}{13}$ and $n^{-2}e\le\frac{1}{3}(325\gamma^2-400\gamma+124).$
    \item If $d=4$ then $\gamma\le\frac{11}{18}$ and $n^{-2}e\le225\gamma^2-275\gamma+\frac{253}{3}.$
 \end{enumerate}
\end{lemma}

Now we can prove the following weak upper bound for $\delta_4.$
\begin{theorem}
  Suppose $G$ is a graph with minimum degree $\frac{49}{52}n+1$. Then the largest $K_4$-free and the largest tripartite subgraphs of $G$ have equal size. Therefore, $\delta_4\le\frac{49}{52}.$
\end{theorem}
\begin{proof}

  Let $G$ be a graph with minimum degree $\delta\geq\frac{49}{52}n+1$. Then $e(G)\ge\frac{1}{2}\delta n$. We suppose that $P_3(G)<K_4f(G)$ and derive a contradiction.

  Let $H$ be a $K_4$-free subgraph of $G$ with $e(H)=K_4f(G)$ maximal, and write $e(H)=tn^2.$ Since $K_4f(G)>P_3(G)\ge\frac{2}{3}e(G)\geq\frac{\delta n}{3}$, we have that $t>\frac{\delta}{3n}$. As we did in the proof of Theorem \ref{general upper bound}, we construct a sequence of graphs $H=H_n,H_{n-1},\ldots,$ where if $H_k$ has a vertex of degree less than or equal to $\frac{8}{13}k$ we delete that vertex to obtain $H_{k-1}.$ Let $\Gamma$ be the final graph of this sequence and write $v(\Gamma)=\alpha n.$ Then $\Gamma$ is $K_4$-free with minimal degree $\delta(\Gamma)>\frac{8}{13}v(\Gamma)$ and $e(\Gamma)\ge e(H)-\frac{8}{13}(n+(n-1)+(n-2)+\ldots+(\alpha n+1))=tn^2-\frac{4}{13}(n+\alpha n+1)(n-\alpha n),$ i.e.
  \begin{equation}\label{upper bound of t}
    n^{-2}e(\Gamma)\ge t-\frac{4}{13}(1+\alpha+\frac{1}{n})(1-\alpha).
  \end{equation}
  As $\delta(\Gamma)>\frac{8}{13}\alpha n,$ by Theorem \ref{homomorphic}, $\Gamma$ is homomorphic to $F_d+K_1$ for some $d\le3.$ Choose $d$ so that $\Gamma$ is not homomorphic to $F_i+K_1$ for any $i<d.$ Now we focus on the evaluation of $d.$

 If $d=3,$ then by Lemma \ref{condition of d} (ii) with $\gamma=\frac{8}{13}$ we have
 \begin{equation*}
   n^{-2}e(\Gamma)\le\frac{4}{13}\alpha^2.
 \end{equation*}
 Together with (\ref{upper bound of t}), we have $t\le\frac{4}{13}\alpha^2+\frac{4}{13}(1+\alpha+\frac{1}{n})(1-\alpha).$ Since $t>\frac{\delta}{3n}$, it follows that $\delta<\frac{12}{13}(1+\frac{1}{n}-\frac{\alpha}{n})n\leq\max\limits_{0\le \alpha\le1}\frac{12}{13}(1+\frac{1}{n}-\frac{\alpha}{n})n=\frac{12}{13}n+\frac{12}{13}$, which contradicts that $\delta\geq\frac{49}{52}n+1$.

  If $d=2,$ then by Lemma \ref{condition of d} (i) with $\gamma=\frac{8}{13}$ we have
 \begin{equation*}
   n^{-2}e(\Gamma)\le\frac{53}{169}\alpha^2.
 \end{equation*}
 Similarly, together with (\ref{upper bound of t}), we have $t\le\frac{53}{169}\alpha^2+\frac{4}{13}(1+\alpha+\frac{1}{n})(1-\alpha).$ Since $t>\frac{\delta}{3n}$, it follows that $\delta<(\frac{3}{169}\alpha^2-\frac{12}{13n}\alpha+\frac{12}{13}(1+\frac{1}{n}))n\leq\max\limits_{0\le \alpha\le1}(\frac{3}{169}\alpha^2-\frac{12}{13n}\alpha+\frac{12}{13}(1+\frac{1}{n}))n=\max\{\frac{12}{13}n+\frac{12}{13},\frac{159}{169}n\}$, which contradicts that $\delta\geq\frac{49}{52}n+1$ as well.

  Therefore the only condition left is that $d=1,$ i.e. $\Gamma$ is $3$-partite (since $\delta(\Gamma)>\frac{8}{13}\alpha n,$ it is obvious that $\Gamma$ can not be bipartite). The number of edges of $G$ incident to vertices in $V(G)\backslash V(\Gamma)$ is
  \begin{align*}
    m=&\sum_{v\in V(G)\backslash V(\Gamma)}d(v)-e(V(G)\backslash V(\Gamma))\geq(1-\alpha)n\delta-{(1-\alpha)n\choose 2}.
  \end{align*}
  Applying Lemma \ref{r-1 partite wiht more edges}, we have
  \begin{align*}
    t =& n^{-2}K_4f(G)>n^{-2}P_3(G)\ge n^{-2}(e(\Gamma)+\frac{2}{3}m) \\
      >& t-\frac{4}{13}(1+\alpha+\frac{1}{n})(1-\alpha)+\frac{2}{3n^2}[(1-\alpha)n\delta-\frac{(1-\alpha)n((1-\alpha)n-1)}{2}].
  \end{align*}
  This gives that $\delta<3n(\frac{25}{78}-\frac{1}{78}\alpha-\frac{1}{78n})$. Because $\delta\geq\frac{49}{52}n+1$, it follows that $\frac{49}{52}n+1<3n(\frac{25}{78}-\frac{1}{78}\alpha-\frac{1}{78n})$, i.e. $\alpha<\frac{1}{2}-\frac{27}{n}$.

  On the other hand, by Tur\'{a}n's theorem, we have $e(\Gamma)\le v(\Gamma)^2/3=\frac{\alpha^2n^2}{3}.$ So by inequality (\ref{upper bound of t}), we have $t\le\frac{1}{3}\alpha^2+\frac{4}{13}(1+\alpha+\frac{1}{n})(1-\alpha)$. Since $t>\frac{\delta}{3n}$, it follows that $\delta<3n(\frac{1}{39}\alpha^2-\frac{4}{13n}\alpha+\frac{4}{13}(1+\frac{1}{n}))$. Thus $\frac{49}{52}n+1<3n(\frac{1}{39}\alpha^2-\frac{4}{13n}\alpha+\frac{4}{13}(1+\frac{1}{n}))$, i.e. $0<n\alpha^2-12\alpha-\frac{n}{4}-1$, which cannot hold for every $\alpha\in[0,\frac{1}{2}-\frac{27}{n})$. This completes the proof.
\end{proof}

\section{An improved bound for $\delta_4$}
In this section, we use the notation $a=b\pm c$ to denote $b-c<a<b+c.$
\begin{lemma}\label{llllemma}
Let $\delta=0.9415$ and suppose $1\geq t\geq\delta/3.$ Then there exists $\epsilon>0$ such that the followings hold with $\gamma=\frac{6t-2(2\delta-1)^2}{9+9t-12\delta}-\epsilon$:

(i) $11/18<\gamma<2t$, and

(ii) $3(2t-\gamma)(2-3\gamma)>(3\gamma-4\delta+2)^2.$\\
Suppose also that $t<((31-32\delta)^2+15)/48.$ Then,

(iii) if $\gamma\leq8/13$, then $t>\frac{1}{3}(325\gamma^2-400\gamma+124)$; and

%

(iv) the inequality
$$
(12\gamma^2-\frac{31}{2}\gamma+5)\alpha^2+\frac{1}{2}\gamma\geq t
$$
has no solution with $0\leq\alpha\leq1.$\\
\end{lemma}

\begin{proof}

Before starting, we write $s=3+3t-4\delta$ and note that $s\geq3(1-\delta)>0.$ Then $\gamma=\frac{2}{3}-\frac{8(1-\delta)^2}{3s}-\epsilon.$

(i) We can compute $\frac{d\gamma}{dt}=8(1-\delta)^2s^{-2}$ so $0<\frac{d\gamma}{dt}\leq8/9.$ Therefore to show that $11/18<\gamma<2t$ it suffices to check for $t=\delta/3.$ Then we may compute $0.6147>\gamma=\frac{8}{9}\delta-\frac{2}{9}-\epsilon>0.6146-\epsilon.$ Since $11/18<0.6112$ and $2t>0.6276$ we have $11/18<\gamma<2t$ for small $\epsilon.$

(ii) We have
$$
(6t-3\gamma)(2-3\gamma)-(3\gamma-4\delta+2)^2=6s\epsilon>0.
$$

Now suppose also that $t<t^*=((31-32\delta)^2+15)/48.$ Then we have $s=3+3t-4\delta<64(\delta-1)^2.$

(iii) Suppose that $\gamma\leq8/13.$ Let $g_1(\gamma)=\frac{1}{3}(325\gamma^2-400\gamma+124).$ Then we can compute
$$
\frac{d^2g_1(\gamma(t))}{dt^2}=\frac{8}{3}(1-\delta)^2s^{-4}(15600(1-\delta)^2-200s)+10400(1-\delta)^2s^{-3}\epsilon.
$$
Since $s<64(1-\delta)^2$ we have $\frac{d^2g_1(\gamma(t))}{dt^2}>0.$ Therefore $t-g_1(\gamma)$ is a concave function of $t.$ To show that the function is positive, it suffices to check the extreme values $t=\delta/3$ and $t=t'$, where $t'=0.3146\pm0.0001$ is the value of $t$ at which $\gamma=8/13.$ We have
$$
\delta/3-g_1(\gamma(\delta/3))=0.0060\pm0.0001+O(\epsilon)>0,
$$
and
$$
t'-g_1(8/13)=0.0069\pm0.0001+O(\epsilon)>0,
$$
for small $\epsilon$, as required.

(iv) $12\gamma^2-\frac{31}{2}\gamma+5=12(\gamma-2/3)(\gamma-5/8)>0.$ So the solution to $(12\gamma^2-\frac{31}{2}\gamma+5)\alpha^2+\frac{1}{2}\gamma\geq t$ is $\alpha\geq r_1=\frac{1}{2}s^{1/2}(1-\delta)^{-1}(64(1-\delta)^2-s)^{-1/2}(s-2(1-\delta))+h(\epsilon)$, where $h(\epsilon)\rightarrow 0$ as $\epsilon\rightarrow 0.$ We omit $h(\epsilon)$ since it does not affect the following computation.   We find
$$
\frac{dr_1}{ds}=\frac{1}{2}(1-\delta)^{-1}s^{-1/2}(64(1-\delta)^2-s)^{-3/2}[32(s+2\delta-2)(1-\delta)^2+s(64(1-\delta)^2-s)]>0.
$$
Numerical computation shows that $r_1=1$ for some $t$ in the range $0.31379\pm0.00001.$ If $t\geq0.3138$ we would have the contradiction $\alpha\geq r_1>1$, so we must have $t<0.3138.$ On the other hand, $t\geq\delta/3>0.3138.$ This contradiction completes the proof.
\end{proof}

\begin{theorem}
  Suppose $G$ is a graph on $n$ vertices with minimum degree $0.9415n$, where $n$ is sufficiently large. Then the largest $K_4$-free subgraph and the largest tripartite subgraph of $G$ have equal size. Therefore, $\delta_4\le0.9415.$
\end{theorem}
\begin{proof}
Let $G$ be a graph on $n$ vertices with minimum degree $\delta n.$ Then $e(G)\geq\frac{1}{2}\delta n^2.$ Suppose that $P_3(G)<K_4f(G),$ we can derive a contradiction when $\delta=0.9415.$ This shows that $\delta_4\leq0.9415.$

Let $H$ be a $K_4$-free subgraph of $G$ with $e(H)=K_4f(G),$ and write $e(H)=tn^2.$ Since $K_4f(G)>P_3(G)\ge\frac{2}{3}e(G)\geq\frac{\delta n^2}{3},$ we have that $t>\frac{\delta}{3}.$ As we did in the proof of Theorem \ref{general upper bound}, we construct a sequence of graphs $H=H_n,H_{n-1},\ldots,$ where if $H_k$ has a vertex of degree less than or equal to $\gamma k$ we delete that vertex to obtain $H_{k-1}.$ Let $\Gamma$ be the final graph of this sequence and write $v(\Gamma)=\alpha n.$ Then $\Gamma$ is $K_4$-free with minimal degree $\delta(\Gamma)>\gamma v(\Gamma)$ and $e(\Gamma)\ge e(H)-\gamma(n+(n-1)+(n-2)+\ldots+(\alpha n+1))=tn^2-\frac{\gamma}{2}(n+\alpha n+1)(n-\alpha n),$ i.e.
  \begin{equation}\label{better upper bound of t}
    \beta\alpha^2:=n^{-2}e(\Gamma)\ge t-\frac{1}{2}\gamma(1-\alpha^2)-O(1/n),
  \end{equation}
  or equivalently $(2\beta-\gamma)\alpha^2\geq2t-\gamma-O(1/n).$ By Lemma \ref{llllemma} (i), $2t-\gamma>0$, so then $2\beta>\gamma$ for sufficiently large $n$, and
  $$
  \alpha^2\ge\frac{2t-\gamma}{2\beta-\gamma}.
  $$
  By Lemma \ref{llllemma} (i), $\gamma>11/18$, so by Theorem \ref{homomorphic}, $\Gamma$ is homomorphic to $F_d+K_1$ for some $d\le4.$ Choose $d$ so that $\Gamma$ is not homomorphic to $F_i+K_1$ for any $i<d.$ Now we focus on the evaluation of $d.$

Suppose that $d=1$, i.e. $\Gamma$ is 3-partite. In this case $e(\Gamma)\leq v(\Gamma)^2/3=\frac{\alpha^2n^2}{3}$ by Tur\'{a}n's theorem, i.e. $\beta\leq1/3.$ So $\gamma<2/3.$ The number of edges of $G$ incident to vertices in $V(G)\backslash V(\Gamma)$ is
  \begin{align*}
    m=&\sum_{v\in V(G)\backslash V(\Gamma)}d(v)-e(V(G)\backslash V(\Gamma))\geq(1-\alpha)n^2\delta-{(1-\alpha)n\choose 2}.
  \end{align*}
  Applying Lemma \ref{r-1 partite wiht more edges}, we have
  \begin{align*}
    t =& n^{-2}K_4f(G)>n^{-2}P_3(G)\ge n^{-2}(e(\Gamma)+\frac{2}{3}m) \\
      \geq& t-\frac{1}{2}\gamma(1-\alpha^2)+\frac{2}{3}[(1-\alpha)\delta-\frac{(1-\alpha)^2}{2}].
  \end{align*}
  Since $\gamma<2/3$, this gives that $\alpha<\frac{3\gamma-4\delta+2}{2-3\gamma}.$ Because $\alpha^2\ge\frac{2t-\gamma}{2\beta-\gamma},$ it follows that
  $$
  \frac{2t-\gamma}{2\beta-\gamma}<(\frac{3\gamma-4\delta+2}{2-3\gamma})^2,
  $$
  and since $\beta\leq1/3$, we have $3(2t-\gamma)(2-3\gamma)<(3\gamma-4\delta+2)^2.$ This contradicts Lemma \ref{llllemma} (ii), so this case leads to a contradiction. Note that if $t\geq((31-32\delta)^2+15)/48$, then we may choose $\gamma=5/8$ to satisfy inequalities (i) and (ii) of Lemma \ref{llllemma}, which immediately gives a contradiction, as by Lemma \ref{Kr-free r-1 partite} we know that $\Gamma$ can only be 3-partite. Therefore we can assume that
  $$
  t<((31-32\delta)^2+15)/48.
  $$

  For the case $d=4$, by Lemma \ref{condition of d} $\gamma\leq11/18$, which contradicts Lemma \ref{llllemma} (i). In the case $d=3$, we get $\beta\leq\frac{1}{3}(325\gamma^2-400\gamma+124)<t$, which again gives the contradiction $\alpha>1.$ Therefore we conclude that $d=2$, i.e. $\Gamma$ has $(F_2+K_1)$-type.

  By Lemma \ref{condition of d} we have $\beta\leq12\gamma^2-15\gamma+5.$ We can rewrite $\alpha^2\ge\frac{2t-\gamma}{2\beta-\gamma}$ as
  $$
  (12\gamma^2-\frac{31}{2}\gamma+5)\alpha^2+\frac{1}{2}\gamma\geq t.
  $$
  However, Lemma \ref{llllemma} (vi) shows the above inequality has no solution with $0\leq\alpha\leq1.$ This completes the proof.

 \end{proof}

\section{Concluding remarks}
In this paper, we generalize the problem of finding conditions on a graph $G$ such that the largest number of edges in a triangle-free subgraph equals the largest number of edges in a bipartite subgraph. We still focus on the minimal degree condition and derive general upper and lower bounds. The following conjecture proposed by Balogh, Keevash and Sudakov \cite{MR2274084} is still open.
\begin{conjecture}[\cite{MR2274084}]
  In any graph on $n$ vertices with minimum degree at least $(3/4 + o(1))n$ the largest triangle-free and the largest bipartite subgraphs have equal size.
\end{conjecture}

\bibliographystyle{abbrv}
\bibliography{noteforslrc}

\begin{thebibliography}{1}

\bibitem{MR340075}
B.~Andr\'{a}sfai, P.~Erd\H{o}s, and V.~T. S\'{o}s.
\newblock On the connection between chromatic number, maximal clique and
  minimal degree of a graph.
\newblock {\em Discrete Math.}, 8:205--218, 1974.

\bibitem{MR1073101}
L.~Babai, M.~Simonovits, and J.~Spencer.
\newblock Extremal subgraphs of random graphs.
\newblock {\em J. Graph Theory}, 14(5):599--622, 1990.

\bibitem{MR2274084}
J.~Balogh, P.~Keevash, and B.~Sudakov.
\newblock On the minimal degree implying equality of the largest triangle-free
  and bipartite subgraphs.
\newblock {\em J. Combin. Theory Ser. B}, 96(6):919--932, 2006.

\bibitem{MR2223630}
A.~Bondy, J.~Shen, S.~Thomass\'{e}, and C.~Thomassen.
\newblock Density conditions for triangles in multipartite graphs.
\newblock {\em Combinatorica}, 26(2):121--131, 2006.

\bibitem{MR2866729}
R.~Chapman, editor.
\newblock {\em Surveys in combinatorics 2011}, volume 392 of {\em London
  Mathematical Society Lecture Note Series}.
\newblock Cambridge University Press, Cambridge, 2011.
\newblock Papers from the 23rd British Combinatorial Conference held at the
  University of Exeter, Exeter, July 3--8, 2011.

\bibitem{MR777160}
P.~Erd\H{o}s.
\newblock On some problems in graph theory, combinatorial analysis and
  combinatorial number theory.
\newblock In {\em Graph theory and combinatorics ({C}ambridge, 1983)}, pages
  1--17. Academic Press, London, 1984.

\bibitem{MR671908}
R.~H\"{a}ggkvist.
\newblock Odd cycles of specified length in nonbipartite graphs.
\newblock In {\em Graph theory ({C}ambridge, 1981)}, volume~62 of {\em
  North-Holland Math. Stud.}, pages 89--99. North-Holland, Amsterdam-New York,
  1982.

\bibitem{MR1264720}
G.~P. Jin.
\newblock Triangle-free graphs with high minimal degrees.
\newblock {\em Combin. Probab. Comput.}, 2(4):479--490, 1993.

\end{thebibliography}
\end{document}